\documentclass[12pt]{amsart}

\def\eps{{\varepsilon}}

\def\Const{{\rm Const}}

\def\Prob{{\mathbb{P}}}

\def\bbA{\mathbb{A}}
\def\bbB{\mathbb{B}}

\def\reals{\mathbb{R}}

\def\RmII{{I\!\!I}}

\def\RmIII{{I\!\!I\!\!I}}

\def\ba{\mathbf{a}}

\def\bc{\mathbf{c}}

\def\bj{\mathbf{j}}

\def\brC{{\bar C}}

\def\brD{{\bar D}}

\def\brc{{\bar c}}

\def\brh{{\bar h}}

\def\breps{{\bar\eps}}

\def\cA{\mathcal{A}}

\def\cB{\mathcal{B}}

\def\cC{\mathcal{C}}

\def\cD{\mathcal{D}}

\def\cF{\mathcal{F}}

\def\cL{\mathcal{L}}

\def\cN{\mathcal{N}}

\def\cR{\mathcal{R}}

\def\cX{\mathcal{X}}

\def\fL{\mathfrak{L}}

\def\fR{\mathfrak{R}}

\def\hD{{\hat D}}

\def\hH{{\hat H}}

\def\hI{{\hat I}}

\def\hc{{\hat c}}

\def\hOmega{{\hat\Omega}}

\def\tA{{\tilde A}}

\def\tcA{{\tilde \cA}}

\def\tB{{\tilde B}}

\def\tD{{\tilde D}}

\def\tg{{\tilde g}}

\def\tnu{{\tilde\nu}}

\def\tOmega{{\tilde{\Omega}}}

\def\tcB{{\tilde{\mathcal{B}}}}

\makeatother

\def\beq{\begin{equation}}
\def\eeq{\end{equation}}

\usepackage[usenames,dvipsnames,svgnames,table]{xcolor}
\newtheorem{theorem}{Theorem}[section]
\newtheorem{proposition}[theorem]{Proposition}
\newtheorem{lemma}[theorem]{Lemma}

\theoremstyle{remark}
\newtheorem{remark}[theorem]{Remark}
\theoremstyle{definition}
\newtheorem{defn}[theorem]{Definition}

\numberwithin{equation}{section}

\author{Dmitry Dolgopyat and P\'eter N\'andori}
\title{Infinite measure renewal theorem and related results.}

\thanks{
The authors thank Ian Melbourne and Dalia Terhesiu for posing this problem to them and for useful
discussions about the renewal theory. They also thank the hospitality of 
Schr\"odinger Institute, Vienna, where this project was started.
The first author was supported by the NSF}

\begin{document}
\maketitle
\section{Introduction}
\label{sec:res}

Mixing plays a central role in studying statistical properties of transformations preserving a probability measure. %For this reason many papers studied mixing properties in probability preserving setting. 
For transformations preserving an infinite measure, mixing is much less understood. In fact, there are
several different generalizations of mixing to infinite measure setting \cite{L1}.  One natural definition is
to require that for a large collection of (nice) sets of finite measure, the probability that the orbit
is in this set at a given time $t$ is asymptotically independent of the initial distribution. This type of mixing
is sometimes called {\em Krickeberg mixing} since it has been studied for Markov chains in \cite{Kr} (other early works on this subject include \cite{Fr, H, Pa}).
This notion of mixing is related to classical renewal theory (\cite{GL}) and to limit distributions
of ergodic sums of infinite measure preserving transformations \cite{DSV}. 
Recently, there was a considerable interest in studying mixing properties of hyperbolic
transformations preserving an infinite measure in both discrete and continuous time settings
(see \cite{AN, BT, G, LT, Mel-I, MT1, MT2, MT3, T1} and references wherein). 

The goal of this note is to describe
a method of deducing mixing for flows from local limit results for  the first return map to an appropriate section. This  approach goes back to \cite{GL} in the independent setting, and in dynamical setting it was
pursued in \cite{AN, DN1, DN2}.  The plan of the paper is the following. In Section \ref{ScAbs}, we explain how to obtain mixing for flows from the local limit theorem and appropriate large deviation bounds
for a section. Section \ref{ScBack} contains tools which are helpful in verifying the abstract conditions
of Section \ref{ScAbs} in specific examples. In particular, in Theorem \ref{ThLD} we obtain sharp large
deviation bounds for quasi-independent random variables.
The last two sections contain specific examples where our assumptions hold. Section \ref{ScInd} is 
devoted to independent random variables. The results of this section are not new but we 
included this example since it allows us to illustrate our approach in the simplest possible setting. 
In particular, it is known since the work of Garcia-Lamperti (\cite{GL}) that in the independent 
case the regular variation of the return time with index $\alpha$ is sufficient for mixing if 
$\alpha>\frac{1}{2}$ but extra assumptions are needed if $\alpha\leq \frac{1}{2}.$ We will present in Section \ref{ScInd} a simple argument to verify our key assumptions 
\eqref{Beta1-2}, \eqref{LLTUpper} for $\alpha>\frac{1}{2},$
and we will see that a more delicate estimate \eqref{EqLLTLDDensity} is required  in the general case.
In Section \ref{sec:LSV}, we show how to verify our assumptions for suspension flows over the
Liverani-Saussol-Vaienti map studied in \cite{LSV}. 

While there is a number of papers dealing with mixing of infinite measure preserving flows 
(see the references at the beginning of the introduction), our approach is more elementary 
than most of the previous works. In particular, we pay a special attention to isolate the key
geometric (quasi-independence) and probabilistic (anticoncetration, exchangebility) ingredients
needed in our method. This could make our method useful also for studying more complicated systems.

\section{Abstract setting}
\label{ScAbs}
\subsection{Results.}
Recall that a function $\fL:\reals^+\to\reals^+$ is called {\em slowly varying} if for each
$h>0,$ $\displaystyle \lim_{t\to\infty} \frac{\fL(ht)}{\fL(t)}=1.$ 
A function $\fR:\reals^+\to\reals^+$ 
which can be represented in the form $\fR(t)=t^\gamma \fL(t)$ with $\fL$ slowly varying is
called {\em regularly varying of index $\gamma.$} Equivalently, for each $h>0,$ 
$\displaystyle \lim_{t\to\infty} \frac{\fR(ht)}{\fR(t)}=h^\gamma.$ 
We refer the reader to \cite{BGT} for
a comprehensive discussion of regularly and slowly varying functions. The properties of these
functions needed in this paper will be recalled in a due time.

Let $f: X\to X$ be a dynamical system preserving a probability
measure $\mu.$ Let $\tau$ be a roof function such that for some slowly varying function $\cL(t)$
\begin{equation}
\label{Tail}
\mu(\tau>t)\sim \frac{\cL(t)}{t^\alpha} \quad \text{for} \quad 0<\alpha<1.
\end{equation}

In particular $\mu(\tau)=\infty.$ Let $g_t:\Omega\to\Omega$
be a suspension flow of $f$ under $\tau.$ It preserves an infinite measure
$\nu$ such that
$$ d\nu(x,s)=d\mu(x) ds. $$
Let $\tau_k(x) = \sum_{i=0}^{k-1} \tau (t^i(x))$.
We are interested in the asymptotics of
$ \nu(\cA\cap g_{-t} \cB) $
for suitable sets $\cA, \cB.$

Recall \cite[\S 1.5.7]{BGT} that there exists a regularly varying 
function $\cR(t)$ of index $\frac{1}{\alpha}$ such that
\begin{equation}
\label{R-T} 
 \lim_{t\to\infty} \frac{t \cL(\cR(t))}{\cR^\alpha(t)}=1. 
\end{equation}
$\cR$ is unique up to asymptotic equivalence: if $\cR_1, \cR_2$ satisfy \eqref{R-T}
then $\displaystyle \lim_{t\to\infty} \frac{\cR_1(t)}{\cR_2(t)}.$ Since all
the results of our paper depend only on asymptotic equivalence class of $\cR$
(that is, the results remain valid if $\cR$ is replaced by an equivalent function)
\cite[\S 1.5.2]{BGT} allows us to assume that $\cR$ is eventually monotone
and we shall do so from now on.  

We start with the following fact.
\begin{proposition}
\label{PrLLT-Ren}
Suppose that for all sets $A,B$ in some algebra $\cX$ we have
\begin{equation}
\label{MixLLT}
\mu\left(x\in A, f^k x\in B, \tau_k(x)\in [t-l, t]\right)\approx
\end{equation}
$$ \brc \rho\left(\frac{t}{\cR(k)}\right) \frac{l}{\cR(k)} \mu(A)\mu(B) $$
uniformly for $t\leq \dfrac{\cR(k)}{\eps}$ where $\rho$ is a bounded
continuous probability density on
$[0, \infty)$ and that there are constants $\beta_1, \beta_2, \beta_3$ such that
\begin{equation}
\label{Beta1-2}
 \beta_2+\frac{\beta_3}{\alpha} <1, \quad \beta_1+\beta_2\alpha+\beta_3=1 
 \end{equation}
and
\begin{equation}
\label{LLTUpper}
\mu\left(\tau_k\in [t, t+l]\right)\leq \frac{C l }{\cL^{\beta_2}(t)  t^{\beta_1} k^{\beta_2}\cR^{\beta_3}(k)} .
\end{equation}
Then for $\cA=A\times [a_1, a_2],$ $\cB=B\times [b_1, b_2]$ we have
\begin{equation}
\label{IMMix}
 \nu(\cA\cap g_{-t} \cB) \cL(t) t^{1-\alpha}\to \hc \mu(A) \mu(B) (a_2-a_1)(b_2-b_1) 
=\hc \nu(\cA) \nu(\cB)
\end{equation}
where
$$ \hc={\brc}{\alpha}  \int_0^\infty \frac{\rho(z)}{z^\alpha} dz. $$
\end{proposition}

\begin{remark} 
\label{remark:port}
Let $\Omega_{\leq M} = \{ (x,s) \in \Omega : s \leq M\} .$
Assume that $X$ is a topological space and that $\cX$ is the collection of sets whose boundary has $\nu$-measure
zero. Then one can prove by standard argument (cf. \cite{Bil})
that \eqref{IMMix} is equivalent to 
$$
\nu (\Phi \Psi \circ g_t) \cL(t) t^{1-\alpha}\to \hc \nu (\Phi) \nu( \Psi)
$$
with either of the following two classes of functions:
\begin{itemize}
\item for any continuous functions $\Phi, \Psi: \Omega \to \mathbb R$ supported on $\Omega_{\leq M}$ for some 
$M<\infty;$
\item for any $\Phi = 1_{\mathcal A}$, $\Psi = 1_{\mathcal B}$, where $\mathcal A, \mathcal B \subset \Omega_{\leq M}$
for some 
$M<\infty$
and $\mu(\partial \mathcal A) = \mu(\partial \mathcal B) = 0$.
\end{itemize}
%Note that we require our test functions to be compactly supported only in the flow direction.
\end{remark}

Assumption \eqref{MixLLT} amounts to the {\em non-lattice} (mixing) local limit theorem.
In fact, the non-lattice assumption is not necessary for mixing of the flow. 
To clarify the situation,
we need some definitions.

\begin{defn}
\label{def:per}
Let $(Y,\lambda,T)$ be a dynamical system. We say that an observable $\varphi: Y \to \mathbb R$ 
is {\it rational} if there is a real number $h$ and two measurable functions 
$\psi: Y \rightarrow \mathbb Z$, $ \mathfrak h:Y \to \mathbb R$ so that
$$
\varphi = h \psi + \mathfrak h - \mathfrak h \circ T.
$$
A function, which is not rational, is  called {\em irrational}.

\smallskip\noindent
We say that $\varphi$ is {\em periodic} if there
exist real numbers $\ba, h$ and two measurable functions 
$\psi: Y \rightarrow \mathbb Z$, $ \mathfrak h:Y \to \mathbb R$ so that
\begin{equation}
\label{EqPeriodic}
\varphi =\ba+ h \psi + \mathfrak h - \mathfrak h \circ T.
\end{equation}
A function, which is not periodic, is called {\em aperiodic}.
\end{defn}

A rational function is clearly periodic with $\ba=0.$ Conversely, suppose that 
\eqref{EqPeriodic} holds {\em and} 
$\frac{\ba}{h}$ is rational, that is $\ba=\frac{ph}{q}. $
In this case, letting $\bar\psi=\frac{\ba+h\psi}{\brh}$ with $\brh=\frac{h}{q}$,
we obtain that $\bar\psi$ is integer valued and 
%$\ba+ h \psi = \bar h \bar \psi$ with $\bar \psi: Y \rightarrow \mathbb Z$ and  
% $\psi$ takes values in $\brh\mathbb Z$ where $\brh=\frac{h}{q},$,
so $\varphi$ is rational.

Thus we have three cases: $\varphi$ can be either aperiodic, periodic irrational or rational.

Recall that if $\tau$ and $\bar \tau$ 
are homologous, i.e. $\bar\tau=\tau+\mathfrak h - \mathfrak h \circ T$,
then the corresponding suspension flows are conjugated by the transformation
$\bar{s}=s+h(x).$

Proposition \ref{PrLLT-Ren} addresses mixing in the case $\tau$ is aperiodic. If $\tau$ is rational 
then $g_t$ is not mixing. Indeed, after the change of variables we can assume that the
roof function belongs to $h\mathbb Z.$ In this case, if the initial condition has $s$ close to $0$,
then $s(t)$ will be close to $h\mathbb Z$ for all times in $h\mathbb{Z}$ so it can not come close
to the set $|s-\frac{h}{2}|\leq \eps$ if $\eps$ is small enough.

It remains to address the periodic irrational case. This is done in Proposition \ref{prop:lattice} below.

Given a function $\mathfrak h: X \rightarrow \mathbb R$ and numbers
$k, w_k$, let $\mathcal F_{k,\mathfrak h,w_k}: X \rightarrow X \times X
\times \mathbb R$ be defined by
$$
\mathcal F_{k,\mathfrak h,w_k}: x \mapsto (x,f^kx,
\tau_k (x) - \mathfrak h (x) + \mathfrak h (f^k(x)) - w_k).
$$

\begin{proposition}
\label{prop:lattice}
Assume that $X$ is a topological space.
%The conclusion of Proposition \ref{PrLLT-Ren} 
\eqref{IMMix} 
remains valid 
for any $A$, $B$ with $\mu (\partial A ) = \mu (\partial B ) = 0$
if \eqref{MixLLT} is replaced by
the following assumption

"There is a 
bounded and continuous function $\mathfrak h: X \rightarrow \mathbb R$ and
constants $\ba, h$ such that $\dfrac{\ba}{h}$ is irrational and 
\begin{equation}
\label{LLTLat}
\lim_{k \rightarrow \infty}
\cR(k) \int \phi d(\mathcal F_{k,\mathfrak h,w_k})_* \mu 
=
\brc \rho(w) \int \phi d (\mu \times \mu \times u),
%\mu\left(x\in A, f^k x\in B, \tau_k(x)\in [t, t-l]\right)\approx
\end{equation}
%$$\frac{\brc h}{k^{1/\alpha}} \rho\left(\frac{t}{k^{1/\alpha}}\right) \mu(A) \mu(B) 
%\Card(j: t\leq \ba k+hj \leq t+l)$$}
for any $\phi \in \mathcal C (X \times X\times \mathbb R)$, compactly
supported in the last coordinate, where
$u$ is $h$ times the counting measure on $h \mathbb Z$ and
$w_k \in \ba k + h \mathbb Z$, 
with $|w_k/ \cR(k) - w|$ bounded. Furthermore, the convergence is uniform
for $w< 1/\varepsilon$."
\end{proposition}

%On the other hand, note that \eqref{IMMix} is clearly false if 
%\eqref{LLTLat} holds but $\dfrac{\ba}{h}$ is rational.

\subsection{Proofs.}

\begin{proof}[Proof of Proposition \ref{PrLLT-Ren}]
\begin{equation}\label{RenewSum}
 \nu\left(\cA\cap g_{-t} \cB\right)=
\end{equation}

$$\int_{a_1}^{a_2} \sum_k \mu\left(x\in A, f^k x\in B, 
\tau_k(x)+b_1\leq a+t\leq \tau_k+b_2\right) da. $$
The last condition can be rewritten as
$$ \tau_k(x)\in [t+a-b_2, t+a-b_1]. $$

Let $\cN(t)$ be the smallest number such that $\cR(\cN(t))\geq t.$
Let us decompose the sum \eqref{RenewSum} as $I+\RmII+\RmIII$ where
$I$ includes the terms with $k< \eps \cN(t) ,$ $\RmIII$ includes the terms
with $k\geq \cN(t)/\eps$ and $\RmII$ comprises the remaining terms. By
\eqref{LLTUpper} and Karamata Theorem (\cite[\S 1.5.6]{BGT})
\begin{equation}\label{errorI}
I\leq \frac{\Const}{\cL^{\beta_2}(t) t^{\beta_1} } \sum_{k=1}^{\eps \cN(t)}  
\frac{1}{k^{\beta_2}\cR^{\beta_3}(k)}
\leq \frac{\Const}{\cL^{\beta_2}(t) t^{\beta_1}} 
\frac{\left(\eps \cN(t) \right)^{1-\beta_2}}{\cR^{\beta_3} (\eps\cN(t))}
\end{equation}
Since $\cR$ is regularly varying we have (see \cite[\S 1.5.7]{BGT}) that
\begin{equation}
\label{RN-T}
 \lim_{t\to\infty} \frac{\cR(\cN(t))}{t}=1. 
\end{equation}
Hence
$$\lim_{t\to\infty} \frac{\cR(\eps\cN(t))}{t}=\eps^{1/\alpha} .$$
Thus

$$
I\leq \Const \eps^{1-\beta_2-\beta_3/\alpha} \frac{ \cN^{1-\beta_2} (t)}
{\cL^{\beta_2}(t) t^{\beta_1+\beta_3}} $$
Comparing \eqref{R-T} and \eqref{RN-T} we obtain
\begin{equation}
\label{N-T}
\lim_{t\to\infty} \frac{\cL(t) \cN(t)}{t^\alpha}=1.
\end{equation}
Therefore
\begin{equation}
\label{I-Neg}
I\leq \Const \frac{\eps^{1-\beta_2-\beta_3/\alpha} }{\cL(t)t^{\beta_1+(\beta_2-1)\alpha+\beta_3}}=
\Const \; \frac{\eps^{1-\beta_2-\beta_3/\alpha}}{\cL(t) t^{1-\alpha}}
\end{equation}
and so $I$ is negligible.

Next, by \eqref{MixLLT} and Karamata Theorem 
$$ \RmIII\leq\Const \sum_{k>\cN(t) /\eps} \frac{1}{\cR(k)} \leq 
\Const \eps^{(1/\alpha)-1}  \frac{\cN(t) }{\cR(\cN(t))} . $$
Using \eqref{RN-T} and \eqref{N-T} 
we see that
$$ \RmIII\leq \Const \eps^{(1/\alpha)-1} \frac{t^\alpha}{\cL(t) t} $$
which is also negligible.

On the other hand by \eqref{MixLLT} we have
\begin{equation}\label{leadingsum}
 \RmII\sim  \brc (a_2-a_1) (b_2-b_1) \mu(A)\mu(B) 
 \sum_{k=\eps \cN(t) }^{\cN(t)/\eps}
\rho\left(\frac{t}{\cR(k)}\right) \frac{1}{\cR(k)}. 
\end{equation}
By regular variation
$$ \frac{t}{\cR(k)}= \frac{\cR(\cN(t))}{\cR(k)} (1+o(1))=
\left(\frac{\cN(t)}{k}\right)^{1/\alpha} (1+o(1)) $$
so the sum in \eqref{leadingsum} is asymptotic to
$$ \frac{1}{t}  \sum_{k=\eps \cN(t) }^{\cN(t)/\eps}
\rho\left(\left(\frac{\cN(t)}{k}\right)^{1/\alpha} \right) 
\left(\frac{\cN(t) }{k}\right)^{1/\alpha} . $$
Let $z_k=\left(\frac{\cN(t)}{k}\right)^{1/\alpha}.$ 
Then
$ z_k-z_{k+1} \approx  \frac{\cN^{1/\alpha}(t)}{\alpha k^{1+1/\alpha}}. $
Writing
$$ \rho\left(\left(\frac{\cN(t)}{k}\right)^{1/\alpha}\right) \left(\frac{\cN(t) }{k}\right)^{1/\alpha} \sim
\rho(z_k) {\alpha } (z_k-z_{k+1}) k=
{\alpha } {\cN(t)}\rho(z_k) \frac{z_k-z_{k+1}}{ z_k^{\alpha} }$$
%\left(\frac{1}{\alpha} \;\; \frac{t}{k^{1+(1/\alpha)}}\right) \;\; \left(\frac{k}{t^\alpha}\right)\rho\left(\frac{t}{k^{1/\alpha}}\right) 
%\alpha t^{\alpha-1} $$
%$$=\alpha t^{\alpha-1} (z_k-z_{k+1}) \rho(z_k) z_k^{-1/\alpha} (1+o(1)),
%$$
we see that the sum in \eqref{leadingsum} is asymptotic to  
\begin{equation}
\label{RhoInt}
\frac{\alpha \cN(t)}{ t}
\int_{L_1(\eps)}^{L_2(\eps)} \frac{\rho(z)}{z^{\alpha}} \; dz 
\end{equation}
where $L_1(\eps)\to 0$ and $L_2(\eps)\to \infty$ as $\eps\to 0.$ 
Combining the estimates for $I,$ $\RmII,$ and $\RmIII$ and using \eqref{N-T} to eliminate 
$\cN(t)$ from  \eqref{RhoInt} 
we obtain the result.
\end{proof}

\begin{proof}[Proof of Proposition \ref{prop:lattice}]
Let 
$$
C(a) = \{ (x,y,z)\in X \times X \times \mathbb R:
x \in A, y \in B, z - \mathfrak h(x) + \mathfrak h(y) \in [a - b_2, a- b_1]
\}
$$
Then
$$
\nu(\mathcal A \cap g_{-t} \mathcal B) \\
= \int_{a_1}^{a_2} \sum_k 
(\mathcal F_{k,\mathfrak h,t})_* \mu (C(a)) da
$$
We decompose this sum into $I+\RmII+\RmIII$ as before, and use the
same estimates for $I + \RmIII$. 
To revisit the computation of $\RmII$,
first observe that $\mu (\partial A ) = \mu (\partial B ) = 0$
implies $(\mu \times \mu \times u) (\partial C(a)) = 0$ and thus
by approximating $1_C$ with continuous functions, we find
that \eqref{LLTLat} implies
$$
\RmII \sim \sum_{k=\varepsilon\cN(t) }^{\cN(t)/\varepsilon}
\frac{\brc}{\cR(k) }  \rho\left(\frac{t}{\cR(k)}\right)
\int_{a_1}^{a_2} (\mu \times \mu \times u) (\{ (x,y,z):
(x,y,z - k\ba ) \in C(a) \} ) da.
$$
Fixing some $Q$ large positive integer and writing
$$
 \sum_{k=\varepsilon \cN(t)}^{\cN(t)/\varepsilon}
= \sum_{i=0}^{(\frac{1}{\eps} - \eps -1) \frac{\cN(t)}{Q}}\;\;
\sum_{k=\eps t^{\alpha} +iQ}^{\eps t^{\alpha} +(i+1)Q -1},
$$
we find that 
\begin{align*}
\RmII &\sim \brc Q 
 \sum_{i=0}^{t^{\alpha} (\frac{1}{\eps} - \eps -1) \frac{1}{Q}}
\rho\left(\frac{t}{\cR(\eps \cN(t)+iQ)}\right) \frac{1}{\cR(\eps\cN(t) +iQ)}\\
&\int_{a_1}^{a_2} \frac{1}{Q}
\sum_{k=\eps t^{\alpha} +iQ}^{\eps t^{\alpha} +(i+1)Q -1}
(\mu \times \mu \times u) (\{ (x,y,z):
(x,y,z - k\ba ) \in C(a) \} ) da
\end{align*}
Since $\frac{\ba}{h}$ is irrational, 
Weyl's theorem implies that the integrand in the second line
of the last display is
$\mu(A) \mu(B) [b_2-b_1](1+o_{Q \rightarrow \infty}(1))$,
uniformly in $a,t, \eps$ and $i$.
Thus $\RmII$ is asymptotic to a Riemann sum and the 
proof can be completed as in the case of Proposition 
\ref{PrLLT-Ren}.
\end{proof}

\begin{remark}
\label{RkMultiSum}
The conclusion of Propositions \ref{PrLLT-Ren} and \ref{prop:lattice}
remain valid if \eqref{LLTUpper}
is replaced by
\begin{equation}
\label{LLTUpper2} 
\mu\left(\tau_k\in [t, t+l]\right)\leq \sum_{j=1}^r 
\frac{C l}{\cL^{\beta_{2,j}} (t) t^{\beta_{1,j}} k^{\beta_{2,j}} \cR^{\beta_{3,j}}(k)} .
\end{equation}
where $\beta_{1,j}$ $\beta_{2,j}$ and $\beta_{3,j}$ satisfy \eqref{Beta1-2} for each $j.$
Indeed, we can replace \eqref{I-Neg} by
$$
I\leq \Const \sum_{j=1}^r \sum_{k=1}^{\eps \cN(t) }\frac{1}{{\cL^{\beta_{2,j}} (t) t^{\beta_{1,j}} k^{\beta_{2,j}} \cR^{\beta_{3,j}}(k)} }
\leq
\Const \eps^{1-\max_j \left(\beta_{2,j}+\frac{\beta_{3,j}}{\alpha}\right)} \frac{t^{\alpha-1}}{\cL(t)}. $$
%if $\beta_2=\beta_2'-\beta_2''$ with $\beta_2''>0$ then
%$$ t^{-\beta_1} k^{-\beta_2}=t^{-\beta_1} k^{-\beta_2'}\times k^{\beta_2''}\leq C
%t^{-\beta_1+\alpha \beta_2} k^{-\beta_2'}$$
%so \eqref{LLTUpper2} is actually equivalent to \eqref{LLTUpper}.
\end{remark}

According to a common terminology, 
$\tau$ {\em satisfies a mixing local limit theorem} if either $\tau$ is aperiodic
and \eqref{MixLLT} holds or $\tau$ is periodic (either rational or irrational) and 
\eqref{LLTLat} holds. The results of this section could be summarized as follows.

\begin{theorem}
\label{ThLLT-Ren}
If $\tau$ is irrational, satisfies  a mixing local limit theorem and \eqref{LLTUpper2},
then \eqref{IMMix}  holds.
\end{theorem}

In other words, if the appropriate local limit theorem and large deviation bounds hold for the
base map, then the special flow is mixing in both aperiodic and periodic irrational cases but
not in the rational case. A similar result holds in the finite measure case (see \cite[Section~2]{DN2}).

In the next sections we provide examples of systems satisfying the 
conditions of Theorem \ref{ThLLT-Ren}.

\subsection{Power tail.} Here we consider an important special case where
the function $\cL$ is asymptotically constant. Thus we assume that there is a 
constant $\bc$ such that

\begin{equation}
\label{TailP}
\mu(\tau>t)\sim \frac{\bc}{t^\alpha} \quad \text{for} \quad 0<\alpha<1.
\end{equation}
In this case one can take $\cR(k)=(\bc k)^{1/\alpha}$ and the statements of Propositions
\ref{PrLLT-Ren} and \ref{prop:lattice}  can be simplified as follows.

\begin{proposition}
\label{PrLLT-RenP}
Suppose that \eqref{TailP} holds and there is a bounded continuous density $\rho$ on $[0, \infty)$
such that either

(i) $\tau$ is aperiodic and
 for all sets $A,B\in \cX$ we have
\begin{equation}
\label{MixLLTP}
\mu\left(x\in A, f^k x\in B, \tau_k(x)\in [t, t-l]\right)\approx
\end{equation}
$$ \brc \rho\left(\frac{t}{(\bc k)^{1/\alpha}}\right) \frac{l}{(\bc k)^{1/\alpha}} \mu(A)\mu(B) $$
uniformly for $t\leq \dfrac{k^{1/\alpha}}{\eps},$ or

(ii) $\tau$ is periodic irrational and there is a 
bounded continuous function $\mathfrak h: X \rightarrow \mathbb R$ and
constants $\ba, h$ such that $\dfrac{\ba}{h}$ is irrational and 
\begin{equation}
\label{LLTLatP}
\lim_{k \rightarrow \infty}
(\bc k)^{1/\alpha} \int \phi d(\mathcal F_{k,\mathfrak h,w_k})_* \mu 
=
\brc \rho(w) \int \phi d (\mu \times \mu \times u),
%\mu\left(x\in A, f^k x\in B, \tau_k(x)\in [t, t-l]\right)\approx
\end{equation}
%$$\frac{\brc h}{k^{1/\alpha}} \rho\left(\frac{t}{k^{1/\alpha}}\right) \mu(A) \mu(B) 
%\Card(j: t\leq \ba k+hj \leq t+l)$$}
for any $\phi \in \mathcal C (X \times X\times \mathbb R)$, compactly
supported in the last coordinate, where $\cF, u$ are as in Proposition \ref{prop:lattice}
and $w_k \in \ba k + h \mathbb Z,$ 
with $|w_k/ (\bc k)^{1/\alpha} - w|$ bounded. Moreover, the convergence is uniform
for $w< 1/\varepsilon$.

Assume in addition that 
\begin{equation}
\label{LLTUpper2P} 
\mu\left(\tau_k\in [t, t+l]\right)\leq \sum_{j=1}^r \frac{C l}{t^{\gamma_{1,j}} k^{\gamma_{2,j}}} .
\end{equation}
where for each $j=1\dots r$
\begin{equation}
\label{Gamma1-2}
 \gamma_{2,j}<1, \quad \gamma_{1,j}+\gamma_{2,j}\alpha=1.
\end{equation}
Then for $\cA, \cB$ satisfying the assumptions of Proposition \ref{PrLLT-Ren} we have
\begin{equation}
\label{IMMixP}
\lim_{t\to\infty} \nu(\cA\cap g_{-t} \cB) t^{1-\alpha}
=\hc \nu(\cA) \nu(\cB).
\end{equation}
\end{proposition}

\section{Tools.}
\label{ScBack}
\subsection{Anticoncentration inequlity.}
Here we obtain a useful {\it a priori} bound.

\begin{proposition}
Suppose that for $|s|\leq \delta$
\begin{equation}
\label{CharBound}
 \left|\mu\left(e^{i s \tau_k}\right)\right|\leq \left(1-c \cL\left(\frac{1}{|s|}\right)|s|^\alpha\right)^k. 
\end{equation}
Then for any interval $I$ of unit size
\begin{equation}
\label{AntiConc} 
 \mu\left(\tau_k\in I\right)\leq \frac{D}{\cR(k)}. 
 \end{equation}
\end{proposition}

\begin{proof}
Without loss of generality we may assume that $\delta\leq 1.$
Denote $Z(s)=\cL(1/|s|) |s|^\alpha.$ Note that by \eqref{R-T},
$Z\left(\dfrac{1}{\cR(k)}\right)=\dfrac{1+o(1)}{k}.$ Next, by the Potter bounds (\cite[\S 1.5.4]{BGT}),
for each $\beta<\alpha$, there is a constant $C(\beta)$ such that for $\dfrac{1}{\cR(k)}\leq |s|\leq \delta$, we have
$$  Z(s)\geq C(\beta) Z\left(\frac{1}{\cR(k)}\right) (s\cR(k))^\beta. $$
Thus \eqref{CharBound} implies that for $s\in [-\delta, \delta]$
\begin{equation}
\label{ZPot}
 \left|\mu\left(e^{i s \tau_k}\right)\right|\leq e^{-\brC (s\cR(k))^\beta}. 
\end{equation}
Let $H(x)=\frac{1-\cos \delta x}{\pi \delta^2 x^2}. $ Then
  $\hH(s)=1_{|s|<\delta} \left(\frac{1}{\delta}-\frac{|s|}{\delta^2}\right).$
  Therefore for each~$a$
  $$ \mu(H(\tau_k-a))\leq 
  \frac{1}{2 \pi} \int_{-\delta}^{\delta} \left|\mu\left(e^{-is\tau_k}\right)\right| \hH(s) ds \leq
  \frac{1}{2 \pi \delta} \int_{-\delta}^{\delta} \left|\mu\left(e^{-is\tau_k}\right)\right| ds \leq \frac{C}{\cR(k)},
  $$
%$$ \leq \frac{1}{2 \pi \delta}  \int_{-\delta}^{\delta} \left(1-c Z(s) \right)^k ds
 %\leq \frac{1}{ \pi \delta}  
 %\int_{0}^{\delta}e^{-\tc k Z(s)} ds \leq \frac{C}{\cR(k)} $$
 where the last step uses \eqref{ZPot}.
% \frac{\brC}{k^{1/\alpha}}. $$
%To estimate this integral 
%Hence
%$$ \mu(H(\tau_k-a))\leq \frac{C}{\cR(k)}. $$
  
  On the other hand $H(x)\geq \frac{47}{96 \pi}$ if $|x|\leq \delta/2.$
  Therefore
  $$ \mu\left(\tau_k-a\in \left[-\frac{\delta}{2}, \frac{\delta}{2}\right)\right)
    \leq \frac{96 \pi}{47} \mu(H(\tau_k-a))\leq \;\; \frac{\hat C}{\cR(k)}. $$
    This proves our claim for intervals of size $\delta.$ Since any interval of unit size can be covered by
    a finite number of intervals of size $\delta$ the result follows. 
\end{proof}  

\subsection{Large deviations.}
\begin{theorem}
\label{ThLD}
Suppose that in addition to \eqref{Tail} the sets $\{(\tau\circ f^j)>t\}$ are {\em quasi-independent}
in the sense that
\begin{equation}
\label{QILarge}
\mu(\tau(f^{j_1} x)>t, \tau(f^{j_2} x)>t)\leq K \mu(\tau(f^{j_1} x)>t) \mu(\tau(f^{j_2} x)>t). 
\end{equation}

Then for $k<\cN(t)$ we have
\begin{equation}
 \label{EqLD}
 \mu(\tau_k(x)>t)\leq \frac{D k\cL(t) }{t^\alpha}. 
\end{equation}
In fact
$$ \lim_{k\to\infty, k/\cN(t)\to 0} \frac{t^\alpha}{k\cL(t)} \mu(\tau_k(x)>t)
=1. $$
\end{theorem}
Since quisiindependence is weaker than independence this results contains in particular 
a large deviation estimate for sums of independent random variables (see \cite{D52,Tk}).

\begin{proof}
Given $H$ to be fixed later let
$$\tau^-=\tau 1_{\tau\leq H},\quad \tau^+=\tau 1_{\tau> H},\quad
\tau_k^\pm=\sum_{j=0}^{k-1} \tau^\pm (f^j x). $$
Then by Karamata theory (see \cite[Thm. 1.6.4]{BGT})
$$ \mu(\tau_k^-)=k\mu(\tau^-)\leq \Const k H^{1-\alpha} \cL(H). $$
Hence by Markov inequality for each $\breps>0$ we have
$$ \mu\left(\tau_k^-\geq \breps t\right)\leq
\Const \frac{k H^{1-\alpha}\cL(H)}{\breps t}. $$
Next
$$ \mu\left(\tau_k^+\geq (1-\breps) t\right)\leq
\mu\left(\tau_k^+1_{A_1}\geq (1-\breps) t\right)+
\mu\left(\tau_k^+1_{A_2}\geq (1-\breps) t\right) $$
where 
$A_1$ is the set where $\tau(f^j x)>H$ for exactly one index $j\in [0, k-1]$ and
$A_2$ is the set where $\tau(f^j x)>H$
for at least two indices $j\in [0, k-1].$  
On $A_1$ we should have $\tau(f^j x)>(1-\breps) t$ so 
$$ \mu\left(\tau_k^+1_{A_1}\geq (1-\breps) t\right)\leq \sum_{j=0}^{k-1} \mu\left(\tau(f^j x)>(1-\breps) t\right)
\leq \frac{k \cL((1-\breps)t)}{[(1-\breps)t]^\alpha}. $$
On the other hand the probability that there are two indices where $\tau\circ f^j$ is large can be estimated via
Bonferroni ineqaulity and \eqref{QILarge} by
$$ \sum_{j_1, j_2} \mu(\tau(f^{j_1} x)>H, \tau(f^{j_2} x)>H)\leq K k^2 \frac{\cL^2(H)}{H^{2\alpha}}. $$
Thus
$$ \mu(\tau_k(x)>t)\leq \Const\left[ \frac{k H^{1-\alpha}\cL(H)}{t}
  +\frac{k^2\cL^2(H)}{H^{2\alpha}}\right]+\frac{k \cL((1-\breps)t)}{[(1-\breps)t]^\alpha}. $$
Choosing $H=(kt)^{1/(1+\alpha)}$ we obtain
$$ \mu(\tau_k(x)>t)\leq \Const\;\;
\dfrac{k^2\cL^2(H) }{kt^{2\alpha/(1+\alpha)}} +\frac{k \cL((1-\breps)t)}{[(1-\breps)t]^\alpha}$$
$$=\Const\left(\dfrac{k}{t^\alpha}\right)^{\frac{2}{1+\alpha}}\cL^2(H)+
\frac{k \cL((1-\breps)t)}{[(1-\breps)t]^\alpha}. $$
This provides the required upper bound on large deviation probability  since $\dfrac{2}{1+\alpha}>1.$

To get the matching lower bound we note that
$$ \mu(\tau_k>t)\geq \mu\left(\max_j \tau(f^j x)>t\right). $$
By Bonferroni inequality the last probability can be estimated from below by

\noindent{\small
$\displaystyle \sum_{j=0}^{k-1} \mu(\tau(f^jx)>t)-
\sum_{j_1, j_2=0}^{k-1}
\mu(\tau(f^{j_1}x)>t, \tau(f^{j_2}x)>t)\geq
\dfrac{k\cL(t)}{t^\alpha}-\dfrac{K k^2\cL^2(t) }{t^{2\alpha}}. $}
\end{proof}

\section{Independent Random Variables.}
\label{ScInd}
Here we consider the case where $t_j=\tau\circ f^{j-1}$ are i.i.d. random variables
having non-lattice distribution.
We will recover a result of \cite{Er}. We note that the optimal results for
the infinite measure renewal theorem for independent random variables
are obtained in \cite{CD}. However, we include the section on independent
random variables in order to illustrate our approach in the simplest possible
setting.

We need to check \eqref{MixLLT} and \eqref{LLTUpper}. Let us first note that 
\eqref{CharBound} is satisfied in our case (see e.g. \cite[Eq. (2.6.38)]{IL})
and hence \eqref{AntiConc} holds.

\subsection{Local Limit Theorem.} 
In case $A=B=X$ \eqref{MixLLT} is proven in \cite{St}.  
Now let $A$, and $B$ some cylinder sets, i.e.
$$ A=\{(t_1, t_2\dots t_n)\in \bbA\}, \quad B=\{(t_{1-n}, \dots t_{-1}, t_0)\in \bbB\}.
$$
Denote $\bbA_R=\bbA \cap B(0, R),$ $\bbB_R=\bbB \cap B(0, R).$ We have
$$ \mu(x\in A, f^k x\in B, \tau_k\in I)=$$
$$\int_{\bbA} \int_{\bbB} \Prob\left(\tau_{k-2n}\in I-\sum_{j=1}^{n}
u_j-\sum_{j=1}^{n} v_j \right)
d \Prob(u_1, u_2, \dots u_{n} ) d\Prob(v_1, v_2, \dots v_{n}). $$
Using the Local Limit Theorem of \cite{St} if 
$$(u_1, u_2, \dots u_{n})\in \bbA_R \quad\text{and}\quad
(v_1, v_2, \dots v_{n})\in \bbB_R,$$  and using the anticoncentration inequality \eqref{AntiConc}
otherwise we obtain \eqref{MixLLT}.

\subsection{Local Large Deviations: $\alpha>\frac{1}{2}$}
Next we prove \eqref{LLTUpper} if $\alpha>\frac{1}{2}.$

%We need
%\begin{proposition}
%\label{LDStable}
%(\cite{Tk})
%$$ \mu(\tau_k\geq t)\leq \frac{Ck}{t^{1/\alpha}}. $$
%\end{proposition}

\begin{proposition}
\label{PrLLTU-Ind}
If $I$ is an interval of unit size, then
 $$ \mu(\tau_k\geq t, \tau_k \in I)\leq \frac{\brC k\cL(t) }{t^{\alpha}\cR(k)}. $$
\end{proposition}
This gives \eqref{LLTUpper} with $\beta_1= \alpha$,
$\beta_2=-1,$ $\beta_3=1.$ Thus $\frac{\beta_3}{\alpha}+\beta_2=\frac{1}{\alpha}-1<1$ iff 
$\alpha>\frac{1}{2}.$

\begin{proof}
$$ \mu(\tau_k\geq t, \tau_k \in I)\leq 
\mu(\tau_{k/2}>t/2, \tau_{k}\in I)+\mu(\tau_k-\tau_{k/2}>t/2, \tau_{k}\in I). $$
By symmetry is suffices to consider the first term
$$\mu(\tau_{k/2}>t/2, \tau_{k}\in I)=\mu(\tau_{k/2}>t/2) \mu( \tau_k-\tau_{k/2}\in I-\tau_{k/2}
|\tau_{k/2}>t/2). $$
The first term is bounded by $\frac{Ck\cL(t)}{t^{\alpha}}$ due to \eqref{EqLD}
and the second term is bounded by $\frac{C}{\cR(k)}$ by \eqref{AntiConc}.
\end{proof} 

\subsection{Local Large Deviations: $\alpha\leq \frac{1}{2}$}

Here we obtain \eqref{LLTUpper2} under an additional assumption.

\begin{proposition}
\label{PrLLTLDDensity}
Suppose that for some (and hence all) $K$ 
\begin{equation}
\label{TailDensity}
\mu(\tau\in [t,t+K])\leq \frac{C(K)\cL(t)}{t^{1+\alpha}} 
\end{equation}
then 
\begin{equation}
\label{EqLLTLDDensity}
\mu(\tau_k>t, \tau_k\in I)\leq 
\frac{C_1 \cL(t) k }{t^{1+\alpha}}+\frac{C_2 }{t} 
\end{equation}
\end{proposition}

This gives \eqref{LLTUpper} with 
$$\beta_{1,1}= 1+\alpha, \quad \beta_{1,2}=-1, \quad \beta_{1,3}=0, \quad
\beta_{2,1}=1, \quad \beta_{2,2}=\beta_{3,2}=0. $$

We note that in case $\tau$ is integer valued,
a stronger result, namely a precise asymptotics,
in the style of Theorem \ref{ThLD}, is proven in \cite{Don}.
It is likely that 
in case \eqref{TailDensity} holds with asymptotic equality,
a similar result holds in the present setting as well.
However,
the one sided bound established here is sufficient for our purposes.

\begin{proof}
We proceed as in the proof of Proposition \ref{PrLLTU-Ind}. Fix large constants $L$ and $r$
and define
$$ \tau_k^+=\sum_{j=0}^{k-1} \tau\circ f^j \chi_{\tau\circ f^j>t}, \quad
\tau_k^-=\sum_{j=0}^{k-1} \tau\circ f^j \chi_{\tau\circ f^j<\cR(k) L^r}, $$
$$ \tau_{k,l}=\sum_{j=0}^{k-1} \tau\circ f^j \chi_{\tau\circ f^j\in [t/L^{l+1}, t/L^{l}]}. $$
Let
$$ A^\pm=\{\tau_k^\pm>t/4, \tau_k\in I\}, \quad
A_l=\{\tau_{k,l}>t/(4\times 2^l), \tau_k\in I\}. $$
To estimate $\mu(A^+)$ let $\bj$ be the first index when $\tau\circ f^j>t.$
Conditioning on the values of $\tau\circ f^j$ for $j\neq \bj$ and using
\eqref{TailDensity} we see that for each $j_0$
$$ \mu(A^+, \bj=j_0)\leq \frac{C\cL(t) }{t^{1+\alpha}}. $$
Summing over $j_0$ we obtain
\begin{equation}
\label{ProbA+} 
\mu(A^+)\leq \frac{C\cL(t) k}{t^{1+\alpha}}.
\end{equation}
On $A_l$ there are at least $m(l)=\left(\frac{L}{2}\right)^l$ indices where
$\tau\circ f^j\in [t/L^{l+1}, t/L^{l}].$
Letting $\bj$ be the first such index, conditioning on the values of $\tau\circ f^j$
for $j\neq \bj$ and using the fact that the probability to have $m(l)-1$ high values for $j\neq \bj$
is $O\left(\left[ck\left(\frac{L^{l +1}}{t}\right)^\alpha \cL_l^*
\right]^{m(l)-1}\right)$ where
$$ \cL_l^*=\max_{s\in [t/L^{l+1}, t/L^l]} \cL(s) $$
we get that
$$ \mu(A_l)\leq q_l:=\brC k^{m(l)} (\cL_l^*)^{m(l)}
\left(\frac{c L^{l +1}}{t}\right)^{m(l)\alpha+1}.$$ 
Using that $\cL(t)$ is slowly varying and so $\cL_{l+1}^*=\cL_{l}^*(1+o(1))$ we obtain
$$ \frac{q_{l+1}}{q_l}=\left(\frac{c(k\cL^*_l)^{ 1/\alpha } 
L^{l +1}}{t}\right)^{\alpha(m(l+1)-m(l))} L^{\alpha m(l+1) +1}(1+o(1))\leq \frac{1}{L}$$
if $\dfrac{\left(k \cL_l^*\right)^{1/\alpha} L^{l+1}}{t}\leq \frac{1}{L^r}.$ 
Let $\ell_{k,t}$ be the smallest number such that 
\begin{equation}
\label{DefEll}
\frac{t}{L^{\ell+1}}<\frac{\left(k \cL_l^*\right)^{1/\alpha} L^r}{2} .
\end{equation}
Then $\sum_{l=1}^{\ell_{k,t}} q_l\leq C q_0$ and so
\begin{equation}
\label{ProbAl}
\mu\left( \bigcup_{l=0}^{\ell_{k,t}} A_l\right)\leq \frac{C \cL(t) k}{t^{1+\alpha}}.
\end{equation}
Note that by \eqref{DefEll}
$$ \frac{t_{\ell_{k,t}}^\alpha}{\cL(t_{\ell_{k,t}})} \times \left(\frac{2}{L^r}\right)^\alpha (1+o(1)) \leq k 
\quad\text{where}\quad t_l=\frac{t}{L^{l+1}}.$$
Since $\cR$ is monotone we have $\cR(k) L^r>t_\ell$ and hence
$$ \{\tau_k>t\}\subset A_+\cup A_-\cup\bigcup_{l=0}^{\ell_{k,t}} A_l. $$

Next
$$ \mu(A^-)\leq \mu\left(\sum_{j=0}^{k/2}\tau\circ f^j \chi_{\tau\circ f^j<\cR(k) L^r}>\frac{t}{8}, 
\tau_k\in I\right)$$
$$+
\mu\left(\sum_{j=k/2+1}^k\tau\circ f^j \chi_{\tau\circ f^j<\cR(k) L^r}>\frac{t}{8}, \tau_k\in I\right). $$
Let us estimate the first term, the second one is similar.
By Markov inequality
$$ \mu\left(\sum_{j=0}^{k/2}\tau\circ f^j \chi_{\tau\circ f^j<\cR(k) L^r}>\frac{t}{8}\right)
\leq\frac{8}{t} \mu\left(\sum_{j=0}^{k/2}\tau\circ f^j \chi_{\tau\circ f^j<\cR(k) L^r}\right) $$
$$ \leq C\frac{k \left(\cR(k)\right)^{1-\alpha}\cL(\cR(k))}{t}$$
where the last step relies on Karamata theory (\cite[Thm. 1.6.4]{BGT}). 

On the other hand by \eqref{AntiConc} 
$$ \mu\left(\tau_k\in I\Big|\sum_{j=0}^{k/2}\tau\circ f^j \chi_{\tau\circ f^j<k^{1/\alpha} L^r}>\frac{t}{8}\right)
\leq \frac{D}{\cR(k)}. $$
Combining the last two displays we obtain
$$ \mu\left(\sum_{j=0}^{k/2}\tau\circ f^j \chi_{\tau\circ f^j<k^{1/\alpha} L^r}>\frac{t}{8}, \tau_k\in I\right)
\leq \frac{\brD}{t} \times \frac{k\cL(\cR(k))}{\cR^\alpha(k)}. 
$$
and hence
$$\mu(A^-)\leq \frac{\tD}{t}\times \frac{k\cL(\cR(k))}{\cR^\alpha(k)}. $$
Now \eqref{R-T} gives
\begin{equation}
\label{ProbA-} 
\mu(A^-)\leq \frac{\hD}{t}.
\end{equation}
Combining \eqref{ProbA+}, \eqref{ProbAl}, and \eqref{ProbA-} we obtain the result.
\end{proof}

\section{LSV map.}
\label{sec:LSV}
\subsection{The result.}

Let $\tilde X = [0,1]$ and $\tilde f: \tilde X \to \tilde X$ be the map  
$$ \tilde f(x)=\begin{cases} x(1+(2x)^r) & \text{if } x\not\in X; \\
2x-1 & \text{if } x\in X \end{cases}
$$
where $X=[1/2, 1].$ 
Consider the special flow $\tilde g_t$ of $\tilde f$ under a roof function
$\tilde \tau$ which is positive and piecewise H\"older, in the sense, that its restrictions on
both $[0, \frac{1}{2})$ and $(\frac{1}{2}, 1]$ are H\"older. Let $\tilde \Omega$ be the phase space of this flow.
By \cite{LSV}, there is a unique (up to scaling) ergodic absolutely continuous $\tilde f$-invariant measure $\tilde \mu$
on $\tilde X$. We assume $r>1$. Then the invariant measure is infinite. Let us normalize it so that 
$\tilde \mu ([1/2,1] ) = 1$. Then $\tilde \nu$, defined by $d \tilde \nu (x,s)= d \tilde \mu(x) ds$
is an infinite invariant measure of $\tilde g_t$. 

\begin{theorem}
\label{propLSV}
Assume that $\tilde \tau$ is irrational. Then for any $\varepsilon >0$ and for any $\tilde A,\tilde B \subset [\varepsilon,1]$ 
with 
$\tilde \mu(\partial \tilde A) = \tilde \mu(\partial \tilde B) = 0$,
$0 < \tilde a_1 < \tilde a_2 < \inf \{ \tau(x), x \in \tilde A\}$,
$ 0 < \tilde b_1 < \tilde b_2 < \inf \{ \tau(x), x \in \tilde B\}$,
$\tilde{\cA}= \tilde A\times [\tilde a_1, \tilde a_2],$ $\tilde{\cB}= \tilde B\times [\tilde b_1, \tilde b_2]$ we have
\begin{equation}
\label{IMMix2}
\tilde  \nu(\tilde{\cA}\cap \tilde g_{-t} \tilde \cB) t^{1-1/r}\to \hc \tilde{\nu}(\tilde \cA) \tilde \nu(\tilde \cB)
\end{equation}
\end{theorem}

Recall from Section \ref{ScAbs} that the irrationality condition
is necessary for \eqref{IMMix2}. We also note that irrationality holds for typical
roof functions $\tilde\tau.$ In particular,
a sufficient condition for the irrationality of $\tilde \tau$ 
%\marginpar{\color{red} $\tilde \tau$ is a priori only measurable. One needs to
%explain why it is (dynamically) continuous.}
is that there are two periodic orbits for the flow $\tilde g$,
the ratio of whose periods is irrational, see, for example, the discussion in
\cite[page 394]{G-Reg}.
%Indeed, assume that for $i=1,2$, $x_i$ satisfies $\tilde f^{k_i}(x_i) = x_i$ with 
%$t_i := \tilde \tau_{k_i}(x_i)$ so that $t_1 / t_2 \notin \mathbb Q$. If
%$\tilde \tau = h \psi+ \mathfrak h - \mathfrak h \circ \tilde f$ was true (with $\psi : \tilde X \rightarrow 
%\mathbb Z$), then we would get 
%$$t_i = \tilde \tau_{k_i} (x_i) = \sum_{j=0}^{k_i -1} h \psi(\tilde f^j (x_i)) \in h \mathbb Z$$
%which is in contradiction with $t_1 / t_2 \notin \mathbb Q$. Thus $\tilde \tau$ is irrational in this case.

To reduce Theorem \ref{propLSV} to our setting we note that $\tilde g$
 can be represented 
as a special flow over the first return map $f: X\to X$.
Specifically, let $R(x) = \min \{ n \geq 1: \tilde f^n (x) \in X\}$ be the first return time to $X$ 
and let $f :X \to X$, $f (x) = \tilde f^{R(x)} (x)$ be the first return map. Let us also extend the definition of $R$ to
$\tilde X \setminus X$ with the same formula (first hitting time).
For a function $\phi: \tilde X \to \mathbb R$, let $\phi_X: X \to \mathbb R$ be defined by
$\phi_X(x) = \sum_{i=0}^{R(x) -1} \phi (\tilde f^i (x))$. Now define the roof function $\tau: X \to \mathbb R_+$
by $\tau = (\tilde \tau)_X$. As before, $g_t$ is the special flow under roof function $\tau$,
$\Omega$ is its phase space and $\nu$ with $d \nu(x,s) = d \mu(x) ds$ is a $g_t$-invariant measure.
As the homomorphism $\iota : \Omega \rightarrow \tilde \Omega,$ $\iota(x,s) = (x,s)$ (with the usual identification
$(x, \tilde \tau (x)) = (\tilde f (x) , 0)$) shows,
$(\tilde \Omega, \tilde \nu, \tilde g_t)$
is a factor of
$(\Omega, \nu, g_t)$. Note that $\iota$ is not invertible.

A cylinder of length $n$ (or shortly, $n$-cylinder) is a set
$$ \{x \in X: R(x)=m_1, R(fx)=m_2\dots R(f^{ n-1} x)=m_n \}. $$
Let us consider the topology on $X$ generated by the cylinder sets.
Let us also fix a metric $d(x,y) = \theta^{s(x,y)}$, where
$s(x,y)$ is the smallest $n$ so that $x$ and $y$ belong to different
cylinders of length $n$ and $\theta <1$ is sufficiently close to $1$. 

\begin{lemma}
\label{lem:rational}
$\tau$ is rational if and only if $\tilde \tau$ is rational.
\end{lemma}
\begin{proof}
Assume that $\tilde \tau = b \tilde \psi + \mathfrak h - \mathfrak h \circ \tilde f$ holds on $\tilde X$
with $\tilde g :\tilde X \to \mathbb Z$.
Then by definition, $\tau = (\tilde \tau)_X = b (\tilde \psi )_X + \mathfrak h - \mathfrak h \circ f$ on $X$. Thus $\tau$
is rational if $\tilde \tau$ is rational. 

Next, assume that 
$\tau = b \psi + \mathfrak h - \mathfrak h \circ f$ on $X$. Let us define the functions $\psi', \mathfrak h ' : \tilde X \to 
\mathbb R$ by
$\psi ' (x) = \psi(x) 1_{\{x \in X\}}$ and $ \mathfrak h '(x) = \mathfrak h(x) 1_{\{x \in X\}}$. Let $\tau': \tilde X \to 
\mathbb R$ be defined by $\tau' =b \psi' + \mathfrak h ' - \mathfrak h ' \circ \tilde f$. Observe that by construction,
$\tau'_X = \tau$ on $X$. In general 
$\tau' $ may not be equal to $ \tilde \tau$ on $\tilde X$, but we have
$\tilde \tau - \tau ' = \mathfrak h '' - \mathfrak h '' \circ \tilde f$ on $\tilde X$, where
$\mathfrak h '' :\tilde X \to \mathbb R$ satisfies
 $\mathfrak h ''(x) = \sum_{i=0}^{R(x) - 1} \tilde \tau ( \tilde f^i(x))- \tau '  ( \tilde f^i(x))$.
We conclude that 
$\tilde \tau = b \psi' + \mathfrak h ' +  \mathfrak h '' - \mathfrak h ' \circ \tilde f- \mathfrak h '' \circ \tilde f$.
Thus $\tilde \tau$
is rational if $ \tau$ is rational. 
\end{proof}

\begin{proposition}
\label{thm1}
Assume that $\tau$ in irrational. Then for any $A,B \subset X$ with $\mu(\partial A) = \mu(\partial B) = 0$,
$\cA=A\times [a_1, a_2],$ $\cB=B\times [b_1, b_2]$ we have
\begin{equation}
\label{IMMix3}
 \nu(\cA\cap g_{-t} \cB) t^{1-1/r}\to \hc \nu(\cA) \nu(\cB).
\end{equation}
\end{proposition}

%We mention that there are other variants of the definition of periodicity and rationality of a function.
%In particular the following holds
%{\color{purple}
%\begin{lemma}\label{lem:Hper}
%Let $Y$ be a Gibbs-Markov map on a probability space $(Y, \lambda)$ and let $\varphi$
%be H\"older on $Y$. In this case, Definition \ref{def:per} is unchanged if
%we also require
%that $\psi$ and $\mathfrak h$ are H\"older.
%\end{lemma}}

First, we prove Proposition \ref{thm1} and then derive Theorem \ref{propLSV} from Proposition \ref{thm1} and Lemma \ref{lem:rational}.

\subsection{Special flow over the induced system.}
\begin{proof}[Proof of Proposition \ref{thm1}]% and Lemma \ref{lem:Hper}]
The proof of Proposition \ref{thm1} is 
divided into two steps. In Step 1, we check that either \eqref{MixLLTP} or \eqref{LLTLatP} holds.
In Step 2, we check that \eqref{LLTUpper2P} holds. By the results 
of Section \ref{sec:res} (with $\alpha = 1/r$), these will imply Proposition \ref{thm1}.
%In Step 1, we also prove Lemma \ref{lem:Hper} for $(X, \mu, f)$ (the proof for general Gibbs-Markov maps is identical).

{\bf Step 1: Checking \eqref{MixLLTP} and \eqref{LLTLatP}.}% and proof of Lemma \ref{lem:Hper}}

First, we note that by \cite{AD}, \eqref{MixLLTP} holds if 
$\tau :X \rightarrow \mathbb R$ is aperiodic. According to \cite{AD}, the function 
$\tau :X \rightarrow \mathbb R$ is aperiodic if there is no $\lambda \in \mathcal S^1$ 
(here, $\mathcal S^1$ is the complex unit circle) and measurable function
$\mathfrak g: X \rightarrow \mathcal S^1$ (other than the trivial $\lambda = 1$, $\mathfrak g = 1$) that
would satisfy
\begin{equation}
\label{per:ad}
e^{it \tau (x)} = \lambda \mathfrak g(x) / \mathfrak g(fx).
\end{equation}
First, we observe that this definition coincides with ours. 
Indeed, if 
$\tau(x) = \ba + \mathfrak h (x) - \mathfrak h (f(x)) + \frac{2 \pi}{t} \psi(x)$ with $\psi: X \rightarrow \mathbb Z$,
then \eqref{per:ad} holds with $\mathfrak g = e^{i t \mathfrak h }$. Conversely, assume that \eqref{per:ad}
holds. Then, by 
 Corollary
2.2 of \cite{AD}, $\mathfrak g$ is H\"older. Next, we define a H\"older
function $\mathfrak h$ which satisfies $e^{i \mathfrak h} = \mathfrak g$.
By the H\"older property of $\mathfrak g$, there is some $K$
such that the oscillation of $\mathfrak g$ on $K$-cylinders is less than
$\sqrt 2$. For any $K$-cylinder $\xi$, fix some $x_{\xi} \in \xi$ and 
define $\mathfrak h (x_\xi)$ as the only number in $[0,2 \pi)$
that satisfies  $e^{i \mathfrak h(x_\xi)} = \mathfrak g (x_\xi)$.
Then for any $y \in \xi$, we choose the unique $\mathfrak h(y)$ which satisfies
$|\mathfrak h(y) - \mathfrak h(x_\xi) | < \pi$ 
and $e^{i \mathfrak h(y)} = \mathfrak g (y)$.
By construction, $\mathfrak h$ is H\"older.
We have now $\tau (x) = \ba + \mathfrak h(x) - \mathfrak h(f(x)) + \psi(x)$,
where $\ba = - \frac{i}{t} \log \lambda$ and 
$\psi : X \rightarrow \frac{2 \pi}{t} \mathbb Z.$
 Hence \eqref{EqPeriodic} holds, so the definition of \cite{AD} is equivalent to ours.
It follows that \eqref{MixLLTP} holds in the aperiodic case. 
%is aperiodic 
%(in the sense defined in \cite{AD}: where aperiodicity for functions with value%s in any locally compact, Abelian,
%second countable topological group is defined).
%Since $\tau$ and $\mathfrak h$ are H\"older, $\psi$ is also H\"older, which proves Lemma 
%\ref{lem:Hper}.

%Let us write $ h = \frac{2 \pi}{t}$. 
%Note that by the irrationality assumption on $\tau$,
%$\ba / h$ is irrational.
Let us now assume that $\tau$ is periodic irrational. 
By the previous paragraph, we can assume that $\mathfrak h $ is H\"older. In order to verify
\eqref{LLTLatP}, it is
enough to consider test functions of the form 
$\phi(x,y,z) = 1_{x \in \mathcal C} 1_{y \in \mathcal D} \phi (z)$,
where $\mathcal C$ and $\mathcal D$ are cylinders and $\phi(z)$ is 
compactly supported. Then \eqref{LLTLatP} follows from Theorem 6.5 of
\cite{AD}, applied to $\psi$, and from the continuous mapping theorem.

{\bf Step 2: Checking \eqref{LLTUpper2P}.}

We note that \eqref{CharBound} is verified in \cite{AD}. In particular \eqref{AntiConc} holds.

The Gibbs-Markov property of $f$ implies that there is a constant $K$ such that if $\cC_1, \cC_2$ are cylinders and
the length of $\cC_1$ is less than $j$ 

\begin{equation}
\label{GM}
\mu(\cC_1\cap f^{-j} \cC_2)\leq K \mu(\cC_1)\mu(\cC_2). 
\end{equation}
Also applying  \eqref{GM} inductively we see that if $\cC_1, \cC_2\dots \cC_l$ are cylinders and 
$j_1, j_2\dots j_{l-1}$ are numbers with length$(\cC_m)\leq j_m$ then
\begin{equation}
\label{GM-l}
\mu(\cC_1\cap f^{-j_1} \cC_2\dots \cap f^{ -j_1 - ... -j_{l-1}} \cC_l)
\leq K^{l-1} \mu(\cC_1)\mu(\cC_2)\dots \mu(\cC_l) 
\end{equation}
 
In particular \eqref{QILarge} holds and so \eqref{EqLD} is satisfied. This allows to check \eqref{LLTUpper}
in case $r<2$ and so $\alpha>\frac{1}{2}.$ In the general case we verify 
\eqref{LLTUpper2P} which is the consequence of the Proposition \ref{PrLocLLD-LSV} 
below.
\end{proof}

\subsection{Local Large Deviations}
\begin{proposition}
\label{PrLocLLD-LSV} 
$ \displaystyle \mu(\tau_k>t, \tau_k\in I)\leq 
\frac{C_1 k}{t^{1+\alpha}}+\frac{C_2 }{t} .$
\end{proposition}

\begin{proof}
Note that \eqref{TailDensity} holds with $\alpha=\frac{1}{r}$ (\cite{LSV}).
We follow the approach of Proposition \ref{PrLLTLDDensity}. In particular, we shall
use the notation of Proposition \ref{PrLLTLDDensity}.
We need to show \eqref{ProbA+}, \eqref{ProbAl}, and \eqref{ProbA-}.

We will say that a cylinder $\cD$ of length $1$ is
{\em high} if $\tau>t/3$ on $\cD.$ We say that the cylinders $\cC_1$ and $\cC_2$ of lengths $m_1$ and $m_2$
respectively
are compatible with
$\cD$ if $m_1+m_2=k-1$ and there is a point 
$x\in \cC_1\bigcap f^{-m_1} \cD\bigcap f^{-m_1 -1} \cC_2$ such that 
$\tau_k(x)\in I.$ Thus
$$ \mu(A^+)\leq \sum_{\cC_1, \cD, \cC_2} \mu\left(\cC_1\bigcap f^{-m_1} \cD\bigcap f^{-m_1 -1} \cC_2\right) $$
where the sum is over compatible cylinders. By Gibbs-Markov property
$$ \mu(A^+)\leq K^2 \sum_{\cC_1, \cD, \cC_2} \mu(\cC_1)\mu(\cD)\mu(\cC_2). $$
Next given $\cC_1, \cC_2, I$ there is an interval $\hI$ of bounded size such that if 
$\cC_1, \cD,$ and $\cC_2$ are compatible, then $\tau(x)\in \hI$ for each $x\in \cD.$
($\hI$ maybe empty if there are no high cylinders compatible with $\cC_1$ and $\cC_2$).
Therefore for each $\cC_1, \cC_2$ 
$$ \sum_{\cD: \cC_1, \cD, \cC_2\text{ are compatible}}\mu(\cD)\leq \frac{C}{t^{1+\alpha}}. $$
On the other hand for each $m_1, m_2$, 
$\displaystyle \sum_{ \cC_j: \text{length}(\cC_j)=m_j} \mu(\cC_j)=1.$
Since $m_1$ can take $k-1$ possible values we get
$$ \mu(A^+)\leq \frac{C k}{t^{1+\alpha}}$$
proving \eqref{ProbA+}. The proof of \eqref{ProbAl} is similar except instead of using three cylinders
to describe the itinerary of $x$ we use $2 m(l)+1$ cylinders, where 
$\cD_1, \cD_2, \dots \cD_{ m(l)}$ are cylinders of length $1$ such that 
$\tau\in \left[\frac{t}{L^{l +1}} , \frac{t}{L^{l}}\right]$ on $\cD_j$ and $\cC_1\dots \cC_{ m(l)+1}$ 
are complimentary cylinders.

It remains to establish \eqref{ProbA-}. We have
$$ \mu(A^-)\leq \sum_{\cC_1, \cC_2} K \mu(\cC_1) \mu(\cC_2) $$
where the sum is over all cylinders of length $k/2$ such that 
$\tau(f^j x) \leq 2L^r k^{1/\alpha}$ for all $x$ in $\cC_1\bigcap f^{-k/2} \cC_2$ and all $j<k$
and either $\tau_{k/2}(x)>t/100$ for all $x\in \cC_1$ or  
$\tau_{k/2}(x)>t/100$ for all $x\in \cC_2.$  To estimate this sum
we note that 
$$\sum_{\cC_1: \;[\sum_{j=0}^{k/2-1} (\tau\circ f^j)\times \chi_{[0, 2 L^r k^{1/\alpha}]}(\tau \circ f^j)]>t/100
\text{ on }\cC_1} \mu(\cC_1)
\leq \frac{C(L, r) k^{1/\alpha}}{t} $$ 
by Markov inequality. On the other hand \eqref{AntiConc} shows thar for each $\cC_1$ 
$$ \sum_{\cC_2: \tau_k(x)\in I\text{ for some }x\in \cC_1\cap f^{-k/2} \cC_2} \mu(\cC_2)\leq
\frac{C}{k^{1/\alpha}} .$$
This shows that 
the contribution of terms where 
$\tau_{k/2}>t/100$ on $\cC_1$ is $O\left(t^{-1}\right).$
Likewise the contribution of terms where 
$\tau_{k/2}>t/100$ on $\cC_2$ is $O\left(t^{-1}\right).$
This proves \eqref{ProbA-}.
\end{proof}

\subsection{Mixing away from the origin.}

Here we deduce Theorem \ref{propLSV} from Proposition \ref{thm1}.
Define $\mathcal A = \iota^{-1} (\tilde{\mathcal A}) \subset \Omega$
and $\mathcal B = \iota^{-1} (\tilde{ \mathcal B}) \subset \Omega$. 
%Note that the condition 
%$\tilde A,\tilde B \subset [\varepsilon,1]$ with $\varepsilon >0$ ensures that 
%$\mathcal A = \cup _{k=1,...K}A_k \times [a_{k,1}, a_{k,2}]$
%and $\mathcal B = \cup _{l=1,...L}B_l \times [b_{l,1}, b_{l,2}]$ with finite $K$ and $L$. 
%Clearly, 
%Theorem \ref{thm1} is valid for such finite unions. 
Since $\iota$ is a homomorphism, we have
$\tilde  \nu(\tilde{\cA}\cap \tilde g_{-t} \tilde \cB) 
=\nu({\cA}\cap  g_{-t}  \cB)$ and 
$\tilde{\nu}(\tilde \cA) = {\nu}( \cA)$, $\tilde{\nu}(\tilde \cB) = {\nu}( \cB)$.
It is easy to check that for any $\tilde {\mathcal E} \subset \tilde \Omega$ with 
$\tilde \nu (\partial \tilde{\mathcal E})
= 0$ (w.r.t. the usual product topology on $\tilde \Omega$) and 
for $\mathcal E = \iota^{-1} \tilde{\mathcal E}$, we have $\mu (\partial \mathcal E)=0$ 
(w.r.t. the product topology on $\Omega$ where the topology in the base is defined by $d$). 
Unfortunately, $\mathcal A$ and $\mathcal B$ are not
subsets of $\Omega_{\leq M}$ in general. Indeed $\Omega_{\leq M}$ is
defined by the requirement that the {\em backward} return time to the base $X$
is bounded, while the condition $\tA, \tB \subset [\eps, 1]$
in Theorem \ref{propLSV} 
allows us to bound {\em forward} return time to $X.$
Since our system is non-invertible the forward and backlward directions
play different roles.
Thus we cannot apply Proposition \ref{thm1} directly and an additional analysis
is required. 

\begin{proof}[Proof of Theorem \ref{propLSV}]

By Lemma \ref{lem:rational}, $\tau$ is irrational.

Let $y_0 = 1$, and $y_{n+1}$ be the preimage of $y_n$ in $[0,1/2]$. Let $x_{n+1}$ be the preimage of $y_n$
in $(1/2, 1]$. The intervals $X_n = (x_{n+1}, x_n] $ form a partition of $X$. In fact, $X_n$ coincides with the 
$1$-cylinder $\{x \in X:  R(x) = n\}$. Furthermore, the intervals $Y_n = (y_{n+1}, y_n]$, $n \geq 1$ form a partition 
of $(0, 1/2]$. Note that $Y_0 = (1/2,1] = X$ (up to measure zero). For $n \geq 0$ let
    $$\tilde \Omega^n = \{ (x,s): x \in Y_n, 0 \leq s \leq \tilde \tau (x)\}, \quad
    \hOmega^N=\bigcup_{n=0}^N \tOmega^n. $$
Since  $\tA$ and  $\tB$ are disjoint from $[0, \varepsilon)$, there is a finite $N=N(\eps)$ so that
  $\tcA, \tcB\subset \hOmega^N.$ So it is sufficient to prove  \eqref{IMMix2} for
  $\tcA, \tcB\subset\hOmega^N$ with $\tnu(\tcA)=0,$ $\tnu(\tcB)=0.$ This will be done in three steps.

  \medskip \noindent
  {\bf Step 1: \eqref{IMMix2} holds for $\tcA, \tcB\subset\hOmega^0$.}

  Indeed,
   in this case $\cA, \cB\in \Omega_{\leq M}$ with $M= \| \tilde \tau \|_{\infty}$,
   so the result follows from
  Proposition \ref{thm1}.

  \medskip\noindent
  {\bf Step 2: \eqref{IMMix2} holds if $\tcA\subset \hOmega^0,$ and $\tcB\subset\hOmega^N$ for some $N.$}

  The proof is by induction on $N.$ The base of induction was done at Step 1. So let us assume that the result holds
  for $\hOmega^{N-1}$ and prove it for $\hOmega^N.$ Let $\tcB_N=\tcB\cap \tOmega^N.$ Since
  $\tcB-\tcB_N\subset \hOmega^{N-1}$ it is enouth to show that the pair $(\tcA, \tcB_N)$ satisfies
  \eqref{IMMix2}.

  Partition $\tcB_N$ into subsets $\tcB_{N,l}$ of small diameter $\delta.$ It suffices to check that for each $l$,
  the pair $(\tcA, \tcB_{N,l})$ satisfies \eqref{IMMix2}. Let
  $$t_l^-=\sup_{(x,s)\in \tcB_{N,l}} \min(t>0: \tg_t(x,s)\in \tOmega^{N-1}), \quad
  \tcB_{N,l}^*=g_{t_l^-} \tcB_{N,l}. $$
      If $\delta$ is sufficiently small, then $\tcB_{N,l}^* \subset \tOmega^{N-1}$
      and the diameter of $\tcB_{N,l}^*$ is less than $\tilde \tau_{min} / 2$. Consequently, the preimage
      of $\tcB_{N,l}^*$ under $\tg_{\mathfrak s}$ is the disjoint union of two sets:
      $\tcB_{N,l}' \subset \tOmega^{N}$ and $\tcB_{N,l}'' \subset \tOmega^{0}$,
      where 
      $$\mathfrak s = \sup \{ s < \tilde \tau_{min} / 2: \exists x \in Y_{N-1}:
      (x,s) \in  \tcB_{N,l}^*\}.$$ 
      Furthermore,
      the preimage of $\tcB_{N,l}' $ under $\tg_{t_l^ - - \mathfrak s}$ is $\tcB_{N,l}$.

       %has at most two preimages under
        %$\tg_{t_l^-}:$ $\tcB_{N,l}$ and $\hat\cB_{N,l}\subset \hOmega^{N-1}$ (which might be empty).
      Thus
      $$\{(x,s)\in \tcA: g_t(x,s)\in \tcB_{N,l}\}=$$
      $$\{(x,s)\in \tcA: g_{t+t_l^-}(x,s)\in \tcB^*_{N,l}\}\setminus
      \{(x,s)\in \tcA: g_{t+t_l^- - \mathfrak s}(x,s)\in \tcB_{N,l}''\}. $$
      By the inductive hypothesis, the RHS is asymptotic to
      $$ \hc t^{1/r-1} \tnu(\tcA) \left[\tnu(\tcB^*_{N,l})-\tnu(\cB_{N,l}'')\right]. $$
      Since $\tg$ is measure preserving, we have
$$ \tnu(\tcB^*_{N,l})-\tnu(\cB_{N,l}'')=\tnu(\tcB_{N,l}') = \tnu(\tcB_{N,l}),$$
      which proves \eqref{IMMix2}.

\medskip\noindent
  {\bf Step 3: \eqref{IMMix2} holds for $\tcA, \tcB\subset\hOmega^N$ for arbitrary $N$.}      
  It suffices to show that for any fixed $\xi>0$, we have
\begin{equation}
\label{Squeeze}
  \hc t^{1/r-1} \nu(\tcA) \nu(\tcB) (1-\xi)\leq
  \tnu(\tcA\cap \tg_{-t} \tcB)\leq \hc t^{1/r-1} \nu(\tcA) \nu(\tcB) (1+\xi)
\end{equation}  
  provided that $t$ is large enough.

  To establish \eqref{Squeeze}, we partition $\tcA$ into sets $\tcA_l$ of small diameter $\delta.$
Let  
$$t_l^+=\sup_{(x,s)\in \tcA_{l}} \min(t>0: \tg_t(x,s)\in \tOmega^{0}), \quad
\tcA_{l}^*=g^{t_l^+} \tcA_{l}. $$
    By bounded distortion, given $\xi$, we can find $\delta(\xi)$ such that if $\delta<\delta(\xi)$,
      then the Jacobian J(x,s) of $\tg_{t_l^+}$ satisfies 
      $$ \left(1-\frac{\xi}{2}\right) \frac{\tnu(\tcA_l^*)}{\tnu(\tcA_l)}
      \leq J(x,s) \leq
      \left(1+\frac{\xi}{2}\right) \frac{\tnu(\tcA_l^*)}{\tnu(\tcA_l)}. $$
Hence
$$ 
\frac{\tnu(\tcA_l)}
{\left(1+\frac{\xi}{2}\right)\tnu(\tcA_l^*)}      
\tnu(\tcA_l^*\cap g_{-(t-t_l^+)} \tcB)
\leq
\tnu(\tcA_l\cap g_{-t} \tcB)
\leq
\frac{\tnu(\tcA_l)}{\left(1-\frac{\xi}{2}\right) \tnu(\tcA_l^*)}      
\tnu(\tcA^*_l\cap g_{-(t-t_l^+)} \tcB). $$
By Step 2,
$$ t^{1-1/r} \tnu(\tcA_l^*\cap g_{-(t-t_l^+)} \tcB)\to \hc \tnu(\tcA_l^*) \tnu(\tcB). $$
Therefore for large $t$
$$  \hc t^{1/r-1} \nu(\tcA_l) \nu(\tcB) (1-\xi)\leq
  \tnu(\tcA_l \cap \tg_{-t} \tcB)\leq \hc t^{1/r-1} \nu(\tcA_l) \nu(\tcB) (1+\xi). $$
  Summing over $l$ we obtain \eqref{Squeeze} completing the proof of the theorem.
\end{proof}

\end{document}